\documentclass[11pt,a4paper,leqno]{amsart}
\usepackage{amssymb,xspace}
\usepackage{amstext}
\usepackage{pdfsync}
\usepackage{hyperref}
\usepackage[mathscr]{eucal}
\theoremstyle{plain}
\usepackage{amsbsy,amssymb,amsfonts,latexsym,eucal,amscd}
\usepackage[dvips]{graphicx}
\usepackage{epsfig}
\usepackage[all]{xy}
\usepackage{pstricks}
\usepackage{mathrsfs}
\newcommand{\tens}[1][]{\mathbin{\otimes_{\raise1.5ex\hbox to-.1em{}{#1}}}}
\newcommand{\lltens}[1][]{{\mathop{\tens}\limits^{\rm \mathbb{L}}}_{#1}}

\newcommand{\Q }{\ensuremath{\mathbb Q}}
\newcommand{\Z }{\ensuremath{\mathbb Z}}

\newcommand{\oo }{\ensuremath{\mathcal{O}}}
\newcommand{\hh }{\ensuremath{\mathcal{H}}}
\newcommand{\ff }{\ensuremath{\mathcal{F}}}
\newcommand{\kk }{\ensuremath{\mathcal{K}}}

\newcommand{\jj }{\ensuremath{\mathcal{J}}}

\newcommand{\eee }{\ensuremath{\mathcal{E}}}

\newtheorem{theorem}{Theorem}[section]

\newtheorem{lemma}[theorem]{Lemma}

\newtheorem{proposition}[theorem]{Proposition}

{\theoremstyle{definition}}

{\theoremstyle{definition}}

{\theoremstyle{definition}}

{\theoremstyle{definition}}

{\theoremstyle{definition}\newtheorem{definition}[theorem]{Definition}}

{\theoremstyle{definition}}

{\theoremstyle{definition}}

\newcommand{\T}[1]{\textrm{#1}}

\newcommand{\M}[1]{\mathrm{#1}}

\newcommand{\ddet}{\ensuremath{\mathrm{det} \,}}

\entrymodifiers={+!!<0pt,\fontdimen22\textfont2>}

\def\apl#1#2#3{#1\mkern -1 mu:\mkern - 6 mu
\xymatrix@C=17pt{#2\!\ar[r]&\!#3}
}

\def\aplexp#1#2#3#4{#1\mkern -1 mu:\mkern - 6 mu
\xymatrix@C=17pt{#2\!\ar[r]^-{#4}&\!#3}
}

\def\aplcourte#1#2#3{#1\mkern -4 mu:\mkern - 8 mu
\xymatrix@C=12pt{#2\!\ar[r]&\!#3}
}

\def\aplpt#1#2#3#4{#1\mkern -4 mu:\mkern - 8 mu
\xymatrix@C=17pt{#2\!\ar[r]&\!#3#4}
}

\author{Julien Grivaux}

\address{CNRS, LATP\\
UMR 6632\\
CMI, Universit\'{e} de Provence\\
39, rue Fr\'{e}d\'{e}ric Joliot-Curie\\
13453 Marseille Cedex 13\\
France.}

\email{jgrivaux@cmi.univ-mrs.fr}
\title{Variation of the holomorphic determinant bundle }

\begin{document}

\begin{abstract}
In this paper, we prove that the Grothendieck-Riemann-Roch formula in Deligne cohomology computing the determinant of the cohomology of a holomorphic vector bundle on the fibers of a proper submersion between abstract complex manifolds is invariant by deformation of the bundle.
\end{abstract}
\vspace*{1.cm}
\maketitle

\begin{section}{Introduction}
The Grothendieck-Riemann-Roch theorem is one of the cornerstones of modern algebraic geo\-metry, and can be stated in its initial form as follows:
\begin{theorem} \cite{BS} For any smooth quasi-projective variety $X$ over a field of characteristic zero, the morphism $\ff \rightarrow \M{ch}(\ff) \, \M{Td}(X)$ from the Grothendieck group $\M{K}(X)$ of coherent sheaves on $X$ to the Chow ring $\M{CH}(X)$ of $X$ commutes with proper push-forward.
\end{theorem}

Since Serre's fundamental papers on coherent sheaves \cite{S1} \cite{S2}, it has become natural and useful to translate results from algebraic to analytic geometry. Concerning the Grothendieck-Riemann-Roch theorem, this has been the object of many researches from early sixties till eighties, starting with the case of analytic immersions \cite{AH} and pursuing with the index theorem for vector bundles and coherent analytic sheaves (see \cite{AS}, \cite{P}, \cite{TT1}, \cite{TT2}). The outcome of these works is O'Brian-Toledo-Tong's proof of the Grothendieck-Riemann-Roch theorem in Hodge cohomology for arbitrary proper holomorphic maps between complex manifolds \cite{OTT}. By completely different methods, Levy \cite{L} succeeded in proving the analogous statement in De Rham cohomology, where the Chern classes are constructed by means of locally-free resolutions in the category of real-analytic coherent sheaves.
\par \medskip
From the middle of the eighties, some new ideas about the Grothendieck-Riemann-Roch theorem emerged after the seminal article of Quillen \cite{Q} introducing canonical hermitian metrics on determinant bundles associated with the cohomology of a  vector bundle on the fibers of a holomorphic submersion (see \cite[Chap.\! VI]{S}). Building on initial results in the case of families of curves (see \cite{Q}, \cite{BF}, \cite{D}), Bismut, Gillet and Soul\'{e} \cite{BGS} proved that, for locally K\"ahler fibrations, the curvature of this determinant bundle is exactly given by the component of degree two of the Grothendieck-Riemann-Roch theorem at the level of differential forms. This theorem has been extended to all degrees in \cite{BK} provided that the higher direct images of the bundle are locally free. Quite recently, Bismut succeeded in removing the K\"{a}hlerianity assumption on the morphism and obtained the following result:
\begin{theorem} \cite{B} \label{Bismut}
For any proper holomorphic submersion $f \colon X \rightarrow Y$ between complex mani\-folds and for any holomorphic vector bundle $\eee$ on $X$ such that the sheaves $\M{R}^i \, f_{*} \,\eee$ are locally free on $Y$, the Grothendieck-Riemann-Roch equality for the couple $(\eee, f)$ holds in the Bott-Chern cohomology ring of $Y$.
\end{theorem}
\par \medskip
On abstract complex manifolds, the finest known cohomology theory where Chern classes exist for holomorphic vector bundles is Deligne-Be\u\i linson cohomology. The ultimate goal of our program would be to prove the Grothendieck-Riemann-Roch theorem in this cohomology. The statement does not immediately make sense even for holomorphic vector bundles, because Chern classes must be defined for the direct images sheaves $\T{R}^i f_* \, \eee$. The problem of defining Chern classes of coherent sheaves in Deligne cohomology is solved on compact manifolds in \cite{G1}. The corresponding Grothendieck-Riemann-Roch theorem is proved only for projective morphisms between complex compact manifolds.
\par \medskip
In this paper, we focus only on  the component of degree two on the base of the Grothendieck-Riemann-Roch theorem in Deligne cohomology for holomorphic vector bundles. Our main result describes completely the variation of the determinant bundle:
\begin{theorem} \label{main}
Let $f \colon X \rightarrow Y$ be a proper holomorphic submersion between complex manifolds $X$, $Y$ and let $(\eee_t)_{t \in \Delta}$ be a holomorphic family of holomorphic vector bundles on $X$ parameterized by the complex unit disc $\Delta$. Then there exists a unique analytic curve $\alpha$ from $\Delta$ to $\M{Pic}^0(Y)$ vanishing at $0$ such that for any $t$ in $\Delta$,
\[
\M{c}_1^{\M{D}}(\alpha(t))= [f_*\{ [\M{ch}^{\M{D}}(\eee_t)-\M{ch}^{\M{D}}(\eee_0)] \, \M{td}^{\M{D}}(\M{T}_{X/Y})\}]^{(2)}
\]
 in the rational Deligne cohomology group $\M{H}^2_{\M{D}} (Y, \Q(1))$. Besides, for any $s$ and $t$ in $\Delta$,
\[
\M{c}_1^{\M{D}} [\ddet \M{R} f_* \, \eee_s]-\M{c}_1^{\M{D}} [\ddet \M{R} f_* \, \eee_t]=\alpha(s)-\alpha(t)
\] in $\M{Pic}(Y)$.
In particular, the class $\M{c}_1^{\M{D}} [\ddet \M{R} f_* \, \eee_t] - [f_*\{ \M{ch}^{\M{D}}(\eee_t) \, \M{td}^{\M{D}}(\M{T}_{X/Y})\}]^{(2)}$ in $\M{H}^2_{\M{D}} (Y, \Q(1))$ is independent of $t$.
\end{theorem}
Even if $\T{Pic}(Y)$ is not a complex manifold, it is possible to give a precise definition of an analytic curve with values in $\T{Pic}(Y)$ that matches with the usual one when the image of $\T{H}^1(Y, \Z_Y)$ is discrete in $\T{H}^1(Y, \oo_Y)$.
\par \medskip
On the one hand, Theorem  \ref{main} is motivated by Teleman's program on the classification of class VII surfaces (see \cite{T1}, \cite{T2}). On the other hand, it is a significant step towards the general Grothendieck-Rieman-Roch theorem in Deligne cohomology, at least in degree two. For instance, we have the following result:
\begin{theorem} \label{second}
Let $\, Y$ and $F$ be complex manifolds such that $F$ is compact, and let $p \colon Y \times F \rightarrow Y$ be the first projection. Then, for any $\mathcal{L}$ in $\M{Pic}^0(Y \times F)$
\[
\M{c}_1^{\M{D}} [\ddet \M{R} p_* \, \mathcal{L}] = [p_*\{ \M{ch}^{\M{D}}(\mathcal{L}) \, \M{td}^{\M{D}}(\M{T}_{Y \times F/Y})\}]^{(2)}
\]
in the rational Deligne cohomology group $\M{H}^2_{\M{D}} (Y, \Q(1))$.
\par \medskip
\end{theorem}
The paper is organized as follows: in \S \ref{det}, we recall the theory of determinants for coherent analytic sheaves (see \cite{KM} and \cite[\S 3]{BGS}) and we prove in Proposition \ref{bc} a base change formula for determinant bundles. We also prove a folklore result (Lemma \ref{classic}) saying that the first Chern class of a coherent sheaf in Hodge cohomology is the same as the first Chern class of its determinant. In \S \ref{del}, we recall the basics of Deligne cohomology (see \cite{EV} and \cite[Chap. 12]{V}) and Lemma \ref{trick} is the main ingredient of the proof of Theorem \ref{main}. Then we discuss analytic curves with values in the Picard group of a complex manifold. Lastly, \S \ref{pf} is devoted to the proofs of Theorem \ref{main} and Theorem \ref{second}.
\par \bigskip
\textbf{Acknowledgments} I wish to thank Andrei Teleman for suggesting the problem solved in this paper, and also Christophe Soul\'e for kindly explaining to me some of the material of \cite{BGS}.
\end{section}

\begin{section}{Holomorphic determinant bundles} \label{det}
Let $\ff$ be a coherent analytic sheaf on a connected complex manifold $X$. The determinant of $\ff$, denoted by $\T{det}\,  \ff$, is a holomorphic line bundle on $X$ defined as follows:
\par \smallskip
-- If $\ff$ is torsion-free, there exists a Zariski-open subset $U$ of $X$ such that $\ff$ is locally free on $U$ and $X\setminus U$ has codimension at least two in $X$. Then the top exterior power of $\ff_{\vert U}$ is a holomorphic line bundle on $U$, which extends uniquely to a line bundle $\T{det}\, \ff$ on $X$.
\par
-- If $\ff$ is a torsion sheaf, let $Z$ be the maximal closed hypersurface contained in $\T{supp}\, \ff$, and let $(Z_i)_{i \in I}$ be the irreducible components of $Z$. For any index $i$, if $\jj_{Z_i}$ denotes the ideal sheaf of $Z_i$, let $m_i$ be the smallest integer $m$ such that $\jj_{Z_i}^{m} \ff$ vanishes generically on $Z_i$. Then $\T{det}\, \ff= \bigotimes_i \oo_X(m_i Z_i)$.
\par
-- If $\ff$ is arbitrary and $\ff_{\T{tors}}$ is its maximal torsion subsheaf, then $\ff / \ff_{\T{tors}}$ is torsion-free and $\T{det}\, \ff=\ddet (\ff / \ff_{\T{tors}}) \otimes \ddet \ff_{\T{tors}}.$
\par \bigskip
The main property of determinants is the following:
for any bounded complex $\kk^{\bullet}$ of coherent sheaves on $X$,  using additive notation for line bundles, we have a canonical isomorphism
\[
\sum_i (-1)^i \,  \ddet {\kk}^i \simeq \sum_i (-1)^i \,  \ddet  \hh^i ({\kk^{\bullet}}).
\]
\par
For any bounded complex $\kk^{\bullet}$ on $X$ with coherent cohomology, the determinant of $\kk^{\bullet}$ is defined by the formula $\ddet \kk^{\bullet} = \sum_i (-1)^i \,  \ddet \, \hh^i ({\kk}^{\bullet})$. Two quasi-isomorphic bounded complexes with coherent cohomology have canonically isomorphic determinants.

\begin{lemma} \label{pullback} Let $\varphi \colon X \rightarrow Y$ be a holomorphic map between connected complex manifolds. Then for any bounded complex $\kk^{\bullet}$ of analytic sheaves on $Y$ with coherent cohomology, if $\mathbb{L} \varphi^*$ denotes the derived pullback by $\varphi$, then $\varphi ^* \,\M{det}\,  \kk^{\bullet} \simeq \M{det}\, (\mathbb{L} \varphi^* \, \kk^{\bullet})$.
\end{lemma}
\begin{proof}
By d\'{e}vissage, we are reduced to prove the lemma when $\kk^{\bullet}$ is a single coherent sheaf in degree zero.
For any Stein subset $U$ of $Y$, let  $\eee^{\bullet}$ be a locally free resolution of $\kk_{|U}$. Then we have canonical isomorphisms
\[ \varphi^* \ddet \kk_{|U} \simeq \varphi ^* [\sum_i (-1)^i \, \ddet (\eee^i)] = \sum_i (-1)^i \, \ddet (\varphi^* \eee^i) = \ddet (\mathbb{L} \varphi^* \, \kk _{|U})
\]
which can be glued together to give a global isomorphism on $X$ between $\varphi^* \ddet \kk$ and $\ddet (\mathbb{L} \varphi^* \, \kk)$.
\end{proof}
\par \bigskip
For any coherent sheaf $\ff$ on $X$, we denote by $\T{c}_i^{\T{H}}(\ff)$ the Chern classes of $\ff$ in $\T{H}^{\,i}(X, \Omega_X^i)$ and by $\T{ch}^{\T{H}}(\ff)$ its Chern character in the total Hodge cohomology ring of $X$.

\begin{lemma} \label{classic}
For any complex manifold $X$ and any coherent analytic sheaf $\ff$ on $X$, we have $\M{c}_1^{\M{H}}(\ff)=\M{c}_1^{\M{H}}(\M{det}\, {\ff})$ in $\M{H}^{\,1}(X, \Omega_X^1)$.
\end{lemma}

\begin{proof}
We can assume that $\ff$ is either a torsion sheaf or a torsion-free sheaf. Besides, it is enough to prove that $\T{c}_1^{\T{H}}(\ff_{\vert U})=\T{c}_1^{\T{H}}(\ddet{\ff_{\vert U}})$ where $U$ is a Zariski open subset of $X$ such that $\T{codim}_{X}(X\setminus U)\geq 2$. Therefore, we have to deal with two different cases:
\par \smallskip
-- First case: the sheaf $\ff$ is a torsion sheaf whose support is a smooth hypersurface. By d\'{e}vissage, it is possible to assume without loss of generality that $\jj_Z \ff=0$, so that $\ddet{\ff}=\oo_X(Z)$. Then, using the Grothendieck-Riemann-Roch theorem in Hodge cohomology for immersions \cite{OTT}, we get $\M{c}_1^{\M{H}}(\ff)=[Z]_{\T{H}}=\M{c}_1^{\M{H}}(\oo_X(Z))=\M{c}_1^{\M{H}}(\M{det}\, {\ff}).$
\par
-- Second case: the sheaf $\ff$ is locally free. Then we know that $\M{c}_1^{\M{H}}(\ff)=\M{c}_1^{\M{H}}(\M{det}\, {\ff})$.
\end{proof}

Let $f \colon X \rightarrow Y$ be a proper holomorphic submersion between two connected complex manifolds $X$ and $Y,$ and let $\eee$ be a locally free sheaf on $X$. By Grauert-Riemenschneider's theorem, the bounded complex $\T{R} f_* \, \eee$ has coherent cohomology.

\begin{definition}
The determinant of the cohomology  $\lambda (\eee, f)$ attached to the couple $(\eee, f)$ is the class of $\ddet (\T{R} f_* \, \eee)$ in $\M{Pic}(Y)$.
\end{definition}
We now state and prove a base change theorem for the determinant of the cohomology. Let $T$ be a complex manifold and $u \colon T \rightarrow Y$ be a closed immersion, and consider the cartesian diagram

\[
\xymatrix {
Z\ar@{}[dr]|{\square} \ar[r]^-{v} \ar[d]_-{g}& X \ar[d]^-{f} \\
T  \ar[r]_{u} & Y
}
\]

\begin{proposition} \label{bc} For any vector bundle $\eee$ on $X,$ $u^* \lambda(\eee, f) = \lambda(v^* \eee, g)$ in $\M{Pic}(T).$
\end{proposition}
\begin{proof}
In the bounded derived category of $Y,$ by using the projection formula twice, we have
\begin{align*}
\T{R}f_*\, \eee \, \lltens[\mathcal{O}_Y] \,u_* \mathcal{O}_T &\simeq \T{R}f_* \,(\eee \, \lltens[\mathcal{O}_X] \, f^* u_* \mathcal{O}_T)
 \simeq \T{R}f_* \,(\eee \, \lltens[\mathcal{O}_X] \, v_* g^* \mathcal{O}_T) \simeq \T{R}f_*\,  \T{R}v_* \, (v^* \eee \, \lltens[\mathcal{O}_Z] \, g^* \mathcal{O}_T)\\
 &\simeq \T{R}u_*\,  \T{R}g_* \, (v^* \eee \, \lltens[\mathcal{O}_Z] \, g^* \mathcal{O}_T) \simeq u_* [\T{R}g_* \, (v^* \eee)].
\end{align*}
This proves that for any integer $i$
\[
\hh^i \,  (\mathbb{L} u^* \,  \T{R}f_* \eee) \simeq \T{R}^i g_* (v^* \eee).
\]
Then we can conclude using Lemma \ref{pullback}.
\end{proof}

\end{section}

\begin{section}{Deligne cohomology} \label{del}
For any complex manifold X and any nonnegative integer $p$, the Deligne complex $\Z_{\T{D}, X}(p)$ is the complex
\[
\xymatrix {
\Z_X \ar[r]^-{(2i \pi)^p} & \oo_X \ar[r]^-{\partial} & \Omega_X^1 \ar[r]^-{\partial}  & \cdots  \ar[r]^-{\partial} &\Omega_X^{p-1},\\
}
\]
\par \medskip
where the sheaf $\Z_X$ sits in degree zero. The integral Deligne cohomology groups of $X$ are defined by the formula $\T{H}^k_{\T{D}}(X, \Z(p))= \mathbb{H}^k(X, \Z_{\T{\T{D},X}}(p))$. Similar definitions hold for the rational Deligne complex $\Q_{\T{D},X}(p)$ as well as for the rational Deligne cohomology groups $\T{H}^k_{\T{D}}(X, \Q(p))$.
\par \medskip
For any locally-free sheaf $\eee$ on $X$, we will denote by $\T{c}_i^{\T{D}}(\eee)$ the rational Chern classes of $\eee$ in $\T{H}^{2i}_{\T{D}}(X, \Q(i))$, and by $\T{ch}^{\T{D}}(\eee)$ the Chern character of $\eee$. Recall that $\T{H}^{2}_{\T{D}}(X, \Z(1))$ is the Picard group of $X$, and that the kernel of the first Chern class morphism
\begin{equation} \label{c1}
\M{c}_1^{\M{D}} \colon \M{Pic} (X) \simeq \M{H}^2_{\M{D}}(X, \Z(1)) \rightarrow \M{H}^2_{\M{D}}(X, \Q(1))
\end{equation}
is exactly the maximal torsion subgroup of $\M{Pic}(X)$.
\par \medskip
There is a natural cup-product in Deligne cohomology (cf \cite{EV}). Besides, for any nonnegative integer $p$, the morphism $\partial \colon \Omega_X^{p-1} \rightarrow \Omega_X^p$ induces a morphism from $\Z_{\T{\T{D},X}}(p)$ to $\Omega_X^p [-p]$. Hence for any nonnegative integer $k$, we obtain a map from $\T{H}^{k+p}_{\T{D}}(X, \Z(p))$ to $\T{H}^k(X, \Omega_X^p)$ which is compatible with cup-products on both sides.
\par \medskip
We now give the key lemma of the proof of Theorem \ref{main}.
\begin{lemma} \label{trick}
Let $\Delta$ be the complex unit disc, let $X$ be a complex manifold and let $\alpha$ be a class in $\M{H}^2_{\M{D}}(X \times \Delta, \Q(1))$ whose image in $\M{H}^1(X \times \Delta, \Omega_{X \times \Delta}^1)$ vanishes. Then there exists a class $\beta$ in $\M{H}^2_{\M{D}}(X, \Q(1))$ such that $\alpha= \M{pr}_1^* \,\beta$.
\end{lemma}
\begin{proof}
Let $\delta$ be the morphism obtained by the composition
\[
\Q_{\T{D}, X \times \Delta}(1) \rightarrow \Omega_{X \times \Delta}^1 [-1] \rightarrow \oo_X \boxtimes \Omega_{\Delta}^1[-1].
\]
Then we have an exact sequence
\[
\xymatrix@C=0.5cm{
&0 \ar[r] &\T{pr}_1^{-1} \, \Q_{\T{D}, X}(1) \ar[r] &\Q_{\T{D}, X \times \Delta}(1) \ar[r]^-{\delta} & \oo_X \boxtimes \Omega_{\Delta}^1[-1] \ar[r] &0
}
\]
which yields the long exact sequence
\[
\xymatrix@C=0.6cm{
&\M{H}^2_{\M{D}}(X, \Q(1)) \ar[r]^-{\T{pr}_1^{*}} &\M{H}^2_{\M{D}}(X \times \Delta, \Q(1)) \ar[r]^-{\delta} & \T{H}^1 (X \times \Delta, \oo_X \boxtimes \Omega_{\Delta}^1).
}
\]
The lemma follows.
\end{proof}
To end this section, we discuss the notion of analytic curve in the Picard group of a complex manifold.
\par \medskip
For an arbitrary $X$, the group $\T{Pic}^0 (X)=\T{H}^1(X, \oo_X)/\T{H}^1(X, \Z_X)$ is generally not a complex mani\-fold since $\T{H}^1(X, \Z_X)$ is not always discrete in $\T{H}^1(X, \oo_X)$. However, we can give the following definition:

\begin{definition}
For any complex manifold $X$, a curve $\gamma \colon \Delta \rightarrow \T{Pic}(X)$ is analytic if the curve $t \rightarrow \gamma(t)-\gamma(0)$ can be lifted to a holomorphic map  with values in $\T{H}^1(X, \oo_X)$.
\end{definition}
Then we have:
\begin{lemma} \label{analytic} $ $
\emph{(i)} If $\gamma \colon \Delta \rightarrow \M{Pic}(X)$ is an analytic curve, then $\gamma$ is entirely determined by $\gamma(0)$ and by the curve $\M{c}_1^{\M{D}} \circ \gamma$ from $\Delta$ to $\M{H}^2_{\M{D}}(X, \Q(1))$, where $\M{c}_1^{\M{D}}$ is given by \emph{(\ref{c1})}.
\par \smallskip
\emph{(ii)} For any class $\alpha$ in $\M{Pic}(X \times \Delta)$ the curve $\gamma \colon \Delta \rightarrow \M{Pic}(X)$ defined by $\gamma(t)=\alpha _{|X \times \{t\}}$ is analytic.
\end{lemma}
\begin{proof}
(i) Assume that $\gamma$ is analytic, that $\gamma(0)=0$ and that for any $t$ in $\Delta$, the image of $\gamma(t)$ in
$\M{H}^2_{\M{D}}(X, \Q(1))$ vanishes. If $\tilde{\gamma}$ is a lift of $\gamma$ from $\Delta$ to $\T{H}^1 (X, \oo_X)$ such that $\tilde{\gamma}(0)=0$, then $\tilde{\gamma}(\Delta)$ lies in the image of $\T{H}^1(X, \Q_X)$ in $\T{H}^1(X, \mathcal{O}_X)$. Therefore $\tilde{\gamma}$ vanishes since it has countable image.
\par
(ii) Let $(U_i)_{i \in I}$ be a Stein cover of $X$ such that all finite intersections of the $U_i$'s are contractible. Then $\alpha$ can be written as a \v{C}hech cohomology class in $\T{H}^2_{\T{D}}(X\times \Delta, \Z(1))$ for the covering $(U_i \times \Delta)_{i \in I}$. This means that we can find a one-cochain $u_{ij}(z,t)$ with values in $\oo_{X \times \Delta}$ and a two-cochain $c_{ijk}(z,t)$ with values in $\Z_{X \times \Delta}$ such that $\alpha$ is represented by the couple $(u_{ij}, c_{ijk})$. Then we have $u_{ij}+u_{jk}+u_{ki}=c_{ijk}$ on $U_i \cap U_j \cap U_k$. Since $c_{ijk}(z,t)=c_{ijk}(z,0)$, we get that for any  $t$ in $\Delta$, $\alpha_{|X \times t}-\alpha_{|X \times 0}$ is represented by the couple $(u_{ij}(z,t)-u_{ij}(z,0), 0)$. The one cochain $(u_{ij}(z,t)-u_{ij}(z,0))$ is a one-cocycle with values in $\oo_X$ depending holomorphically on the variable $t$. This proves the result.
\end{proof}

\end{section}

\begin{section}{Proof of the theorems} \label{pf}

\textit{Proof of Theorem \ref{main}}

\par \medskip
Let us consider the line bundle $\lambda(\eee, f \times \T{id})$ on $Y \times \Delta$. By Proposition \ref{bc}, its restriction to a slice $Y \times \{ t \}$ is $\lambda(\eee_t, f).$ Using Lemma \ref{analytic},
we obtain that the curve $t \rightarrow \lambda(\eee_t, f)$ is analytic.
\par \medskip
Let us now consider the class $\alpha$ in $\T{H}^2_{\T{D}}(Y\times \Delta, \Q(1))$ defined by
\[
\alpha= \T{c}_1^{\T{D}}(\lambda(\eee, f \times \T{id}))-\bigl [(f \times \T{id})_* \bigl (\T{ch}^{\T{D}}(\eee) \, \T{td}^{\M{D}}(\T{T}_{X \times \Delta / Y \times \Delta}) \bigr) \bigr ]^{(2)}.
\]
By Lemma \ref{classic},
\[
\T{c}_1^{\T{H}}(\lambda(\eee, f \times \T{id}))=\sum_p (-1)^p\,  \T{c}_1^{\T{H}}\, [\T{R}^p (f \times \T{id})_* \,\eee]
\] in $\T{H}^1(Y \times \Delta, \Omega^1_{Y \times \Delta})$. Therefore, thanks to the Grothendieck-Riemann-Roch theorem in Hodge cohomology \cite{OTT}, $\alpha$ maps to zero in $\T{H}^1(Y \times \Delta, \Omega^1_{Y \times \Delta})$. By Lemma \ref{trick}, we obtain that for all $t$ in $\Delta$, $\alpha_{|Y \times \{t\}}=\alpha_{|Y \times \{0\}}$ in $\T{H}^2_{\T{D}}(Y, \Q(1))$.
\par \medskip
It remains to compute $\alpha_{|Y \times \{t\}}$. If we denote by $[\, . \,]_{\T{D}}$ the Deligne cohomology class of an analytic cycle, for any class $\beta$ in the rational Deligne cohomology ring of $X \times \Delta$, we have
\begin{align*}
\{ (f \times \T{id})_* \, \beta \}_{|Y \times \{t\}}&=\T{pr}_{1 *} \{(f \times \T{id})_* \, \beta \, . \, [Y \times \{t\}]_{\T{D}}\}\\
&=\T{pr}_{1 *} (f \times \T{id})_* \{ \beta \, .\,  (f \times \T{id})^* \,[Y \times \{t\}]_{\T{D}}\} \\
&=f_* \ \T{pr}_{1 *}  \{ \beta \, .\,  [X \times \{t\}]_{\T{D}}\} \\
&=f_* (\beta_{|X \times \{t\}}).
\end{align*}
Therefore we get
$\alpha_{|Y \times \{t\}}= \T{c}_1^{\T{D}}(\lambda(\eee_t, f))- [f_* \{\T{ch}^{\T{D}}(\eee_t) \, \T{td}^{\M{D}}(\T{T}_{X / Y }) \} ]^{(2)}.$
Remark now that for all $t$ in $\Delta$,
\[
\T{c}_1^{\T{D}}(\lambda(\eee_t, f))-\T{c}_1^{\T{D}}(\lambda(\eee_0, f))= [f_* \{ [\T{ch}^{\T{D}}(\eee_t)-\T{ch}^{\T{D}}(\eee_0)] \, \T{td}^{\M{D}}(\T{T}_{X / Y }) \} ]^{(2)}.
\]
\par \smallskip
Thus the curve $t \rightarrow  [f_* \bigl ([\T{ch}^{\T{D}}(\eee_t)-\T{ch}^{\T{D}}(\eee_0)] \, \T{td}^{\M{D}}(\T{T}_{X / Y }) \bigr) ]^{(2)}$ in $\T{H}^2_{\T{D}}(Y, \Q(1))$ can be lifted to the analytic curve $\alpha \colon t \rightarrow \lambda(\eee_t, f)-\lambda(\eee_0, f)$ in $\M{Pic}^0 (Y)$. This finishes the proof.
\par
\hspace{15.2cm} $\qed$
\par \medskip
\textit{Proof of Theorem \ref{second}}
\par \medskip
It is a direct consequence of Theorem $\ref{main}$. Indeed, since $\mathcal{L}$ lies in $\T{Pic}^0(Y \times F)$, there exists a holomorphic family of holomorphic line bundles joining $\oo_{Y \times F}$ to $\mathcal{L}$. Thanks to Theorem \ref{main}, we are reduced to the case $\mathcal{L}=\oo_{Y \times F}$. In this case, Theorem \ref{second} is straightforward since
\[
\ddet \T{R} p_* (\oo_{Y \times F}) \simeq \ddet [\oo_Y \lltens[\mathbb{C}_Y] \M{R \Gamma} (F, \oo_F)] \simeq \oo_{Y}
\]
and
$p_*[\M{td}^{\M{D}} \, (\M{T}_{Y \times F /Y })]=p_*[1 \boxtimes \M{td}^{\M{D}} \, (F)]$ is concentrated in degree zero.
\par
\hspace{15.2cm} $\qed$
\end{section}

\bibliographystyle{plain}
\bibliography{biblio}

\end{document}